\documentclass[12pt]{amsart}

\usepackage{latexsym, amsfonts, amsmath, amssymb}

\newtheorem{thm}{Theorem}[section]

\newtheorem{cor}[thm]{Corollary}
\newtheorem{prop}[thm]{Proposition}
\newtheorem{lem}[thm]{Lemma}

\numberwithin{equation}{section}

\newcommand{\om}{\omega}
\newcommand{\ot}{\otimes}

\newcommand{\lan}{\langle}
\newcommand{\ran}{\rangle}

\newcommand{\eps}{\varepsilon}
\newcommand{\vs}{\vskip 3pc}
\newcommand{\ds}{\displaystyle}

\def\CC{{\mathcal C}}

\def\CY{{\mathcal Y}}
\def\CD{{\mathcal D}}

\def\BN{{\mathbb N}}

\def\Bk{{\Bbbk}}

\newcommand{\num}{\refstepcounter{thm}}

\begin{document}

\title{A $q$-Identity Related to a Comodule}
\dedicatory{Dedicated to Mia Cohen, coauthor and friend, on the occasion of her retirement}
\author{A. Jedwab} 
\address{University of Southern California, Los Angeles, CA90089-1113}
\email{jedwab@usc.edu}
\author{S. Montgomery}
\address{University of Southern California, Los Angeles, CA 90089-1113}
\email{smontgom@math.usc.edu}
\thanks{Both authors were supported by NSF grant DMS 07-01291}

\maketitle
\setcounter{section}{0}

\section{Introduction} 

In this paper we show that a certain algebra being a comodule algebra over the Taft Hopf algebra 
of dimension $n^2$ is equivalent to a set of identities related to the $q$-binomial coefficient, 
when $q$ is a primitive $n^{th}$ root of 1. We then give a direct  combinatorial proof of these identities. 
To be consistent with the usual notation for the Taft algebra, 
we will write $q = \om$ for our $n^{th}$ root of 1.

Let $\Bk$ be a field of characteristic 0 which contains a primitive $n^{th}$ root of 1, $\om$. 
Consider the algebra $A=A_n(\om) = \Bk[z]/(z^n-\om)$. It was proposed by Cohen, Fischman, and the second author \cite{CFM} that $A$ is a right 
$H$-comodule for the Taft Hopf algebra $H = T_{n^2}(\om)$ of dimension $n^2$, for a particular map 
$\rho: A \to A \ot H$.  \cite[Proposition 2.2(d)]{CFM} proved that $\rho$ is a comodule map when 
$n \leq 4$. However the question for general $n$ was left open, since the general case 
seemed to lead to some rather complicated identities. 

The comodule problem was later solved for arbitrary $n$ in \cite{MS} by indirect means: 
it was shown there that $A$ is a module for the Drinfel'd  double $D(H)$, giving an action 
of the dual $(H^*)^{cop}$ on $A$. This action dualizes exactly to the \cite{CFM} coaction of $H$. 
Moreover \cite{MS} show that $A$ is always a Yetter-Drinfeld module for $H$; this had been proved in \cite{CFM} for $n \leq 4$. 

The question was raised as to whether a direct proof of the comodule property for $A$ via $\rho$ 
could be given, by determining precisely the identities involved (see \cite[p. 357]{MS}). In Theorem 
\ref{main} we determine exactly the identities needed, using the $q$-binomial coefficient with $q = \om.$ 
In Theorem \ref{comain} we then give a combinatorial proof of the identities. This gives an alternative to the methods of \cite{MS}. 

In Section 5 we also show directly that our algebra $A = A_n(\om)$ is in the category 
$^H_H{\mathcal Y \mathcal D}$ of Yetter-Drinfel'd modules for $H$, for any $n$, using the form of 
the comodule map $\rho$.  As a consequence $A$ is always a commutative algebra in the category $^H_H{\mathcal Y \mathcal D}$, answering another question  of \cite{CFM}.

Finally in Section 6 we discuss in more detail the dual approach of \cite{MS}. 

\newpage

%--------------------------------------------Sec 2: Preliminaries--------------------

\section{Preliminaries}

We let $H$ denote the Taft Hopf algebra of dimension $n^2,$ that is 
$$H=T_{n^2}(\om) = \Bk \langle x,g | x^n=0,\ g^n=1,\ xg=\om gx \rangle,$$
where $\omega$ is a fixed primitive $n^{th}$ root of 1, 
with Hopf structure given by
$$\Delta (g)=g\otimes g,\ \Delta (x)=x\otimes 1+g\otimes x$$
$$ \epsilon(g)=1,\ \epsilon (x)=0, \ S(g)=g^{-1}\  \mbox{and}\ S(x)=-g^{-1}x.$$
We also need some well-known facts about $q$-binomial coefficients \cite{K}. Recall that 
$${{b \choose k}_q} := \ds\frac{(b)!_q}{(k)!_q \ (b-k)!_q}, \text{ where } (b)!_q :=
\ds{\frac{(q-1)(q^2-1)\cdots(q^b-1)}{(q-1)^b}}.
$$

So for $k,s\in \BN,$  
$${{k+s \choose k}_q}=\frac{(1-q)\cdots(1-q^s)(1-q^{s+1})\cdots (1-q^{s+k})}{(1-q)\cdots (1-q^s)(1-q)\cdots (1-q^k)}=\frac{(1-q^{s+1})\cdots (1-q^{s+k})}{(1-q)\cdots (1-q^k)}.$$

\begin{lem}
Given $x,g\in H$ as above and $b\in \mathbb{N}$, 
$$\Delta(x^b)= \sum_{k=0}^{b} {{b \choose k}_\omega} \ g^kx^{b-k}\otimes x^k.$$
\end{lem}

\begin{proof} Since $\Delta(x^b) = (\Delta (x))^b = (x\otimes 1+g\otimes x)^b,$ the lemma follows from the $q$-binomial theorem \cite[IV.2.2]{K}, as follows: in \cite{K}, the theorem is stated for $(x+y)^b$, 
where $yx = q xy.$ Here we 
replace $x$ by $g \ot x$, $y$ by $x \ot 1$ and $q$ by $\omega.$
%We prove this formula by induction on $b.$ For $b=1,$$$ \sum_{k=0}^{b}{{b \choose k}_\omega} \ g^kx^{b-k}\otimes x^k={{1 \choose 0}_\omega} g^0x^{1}\otimes x^0+{1 \choose 1}_\omega g^1x^0\otimes x^1=x\otimes 1+g\otimes x=\Delta (x).$$
%Now suppose the formula holds for $b=n.$ Then
%\begin{eqnarray*}
%\Delta (x^{n+1}) &=& (x\otimes 1+g\otimes x)(\sum_{k=0}^{n}{n \choose k}_\omega g^kx^{n-k}\otimes x^k)\\
%                             &=& \sum_{k=0}^n{n \choose k}_\omega xg^kx^{n-k}\otimes x^k + \sum_{k=0}^n{n \choose k}_\omega g^{k+1}x^{n-k}\otimes x^{k+1}\\
%                             &=& \sum_{k=0}^n\omega ^k{n \choose k}_\omega g^kx^{n+1-k}\otimes x^k + \sum_{i=1}^{n+1}{n\choose i-1}_\omega g^ix^{n+1-i}\otimes x^i\\
%                             &=& x^{n+1}\otimes 1 + \sum_{k=1}^n (\omega^k{n\choose k}_\omega + {n\choose {k-1}}_\omega)g^k x^{n+1-k}\otimes x^k + g^{n+1}\otimes x^{n+1}\\
%                             &=& \sum_{k=0}^{n+1}{{n+1}\choose k}_\omega g^kx^{n+1-k}\otimes x^k.\\
%\end{eqnarray*} 
\end{proof}

\begin{cor}\label{delta} 
For any $a,b\in \mathbb{N},$ 
\begin{equation*}
\Delta (x^bg^a) = \sum_{k=0}^{b}w^{-k(b-k)}{b \choose k}_\omega x^{b-k}g^{k+a}\otimes x^kg^a.
\end{equation*}
\end{cor}

\begin{proof}
\begin{eqnarray*}
\Delta {(x^bg^a)} &=& \Delta (x^b)\Delta(g^a)
			   = \sum_{k=0}^{b} {{b \choose k}_\omega} \  g^kx^{b-k}g^a\otimes x^kg^a\\
			   &=&\sum_{k=0}^{b}w^{-k(b-k)} {{b \choose k}_\omega} \  x^{b-k}g^{k+a}\otimes x^kg^a.	
\end{eqnarray*}
\end{proof}

%I dont think we'll need this, but kept it for now just in case:
%We also have 
%\begin{eqnarray*}
%(id\otimes \Delta)\Delta(x^b) &=& \sum_{k=0}^{b}{b \choose k}_\omega g^kx^{b-k}\otimes \Delta(x^k) \\
				 	     %&=& \sum_{k=0}^b{b\choose k}_\omega g^kx^{b-k}\otimes (\sum_{l=0}^k{k\choose l}_\omega g^lx^{k-l}\otimes x^l)\\ 	
%\end{eqnarray*}					     				

%---------------------------------------------------3. The comodule algebra-----------------------------------------------

\section{The comodule algebra for $H$}

As noted in the introduction,  \cite{CFM} proposed that $A$ will be an $H$-comodule. 
We let $u$ denote the coset  $ z +  I$, where $I = (z^n-\om)$,  and thus 
$\{1, u, u^2, \ldots, u^{n-1} \}$ will be a basis for $A.$

For our given root of unity $\om$, we define 
\begin{equation}\num \label{a_i}
a_i:=(\omega-1)^i\omega^{\frac{i(i+1)}{2}}.
\end{equation}

The explicit  coaction $\rho :A\to A\ot H$ is now defined by 
\begin{equation}\num \label{coaction}
\rho(u)=\sum_{i=0}^{n-1}a_ix^ig^{-(i+1)}\otimes u^{i+1}.
\end{equation}

We must prove that 
\begin{equation}\num \label{comod} (id\otimes \rho)\rho=(\Delta \otimes id)\rho.  
\end{equation}

Now \cite{CFM} showed that $\rho(u)^n = \om 1$, and thus  $\rho$ is a homomorphism since 
$u^n = \om 1$. Since $\Delta$ is also a homomorphism and the powers of $u$ are a basis for $A$, 
it will suffice to check that Equation (\ref{comod}) holds when applied to the element $u$.

Evaluating Equation (\ref{comod}) on $u,$ we obtain the new equation:
\begin{equation}\num \label{identity}
\sum_{s=0}^{n-1}a_sx^sg^{-(s+1)}\otimes \rho(u)^{s+1}=\sum_{m=0}^{n-1}a_m\Delta(x^mg^{-(m+1)})\otimes u^{m+1},
\end{equation}
where by Corollary \ref{delta}, the right hand side is 
$$\sum_{m=0}^{n-1}a_m\left( \sum_{k=0}^{m}\omega^{-k(m-k)}{m \choose k}_\omega x^{m-k}g^{k-(m+1)}\otimes x^kg^{-(m+1)}\right) \otimes u^{m+1}.$$

In order to compute the left hand side of (\ref{identity}), we need to find an explicit formula for $\rho(u)^s$ for any $1\leq s\leq n.$
We start with an auxiliary lemma:
\begin{lem}\label{a_rLemma}
(i) Given $r,s \in \mathbb{N},\ a_ra_s=a_{r+s}\omega^{-rs}$ and more generally 
$$\displaystyle \Pi_{i=1}^ta_{r_i}=a_{(\sum_{i=1}^{t}r_i)}\omega^{-\sum_{j<i}r_ir_j}$$

(ii) For all $1\leq i\leq n-1$, 
$$(\sum_{j=0}^i\om^{j-i}) a_i+a_{i-1}=\om^{i+1}a_{i-1}.$$
\end{lem}

\begin{prop} \label{power}
For any $\ds 1\leq s\leq n,$ 
$$\rho(u)^s=\sum_{k=0}^{n-1}a_k \left( \sum_{\{0 \leq i_1, \ldots, i_s\leq k \ | \ \sum_{j=1}^si_j=k \} }\omega^{\sum_{j=2}^si_j(j-1)} \right) x^k g^{-(k+s)}\otimes u^{k+s}.$$ 
\end{prop}

\begin{proof}
Let $\ds \rho(u)_j=\sum_{i_j=0}^{n-1}a_{i_j}x^{i_j}g^{-(i_j+1)}\otimes u^{i_j+1}$ denote the $j$-th copy of $\rho(u)$ in $\rho(u)^s.$  As we multiply one term from each of the $s$ factors $\rho(u)_j$ in $\rho(u)^s,$ we obtain a sum of terms of the form 
\begin{equation}\num \label{term}
a_{i_1}\cdots a_{i_j}\cdots a_{i_s}x^{i_1}g^{-(i_1+1)}\cdots x^{i_j}g^{-(i_j+1)}\cdots x^{i_s}g^{-(i_s+1)}\otimes u^{i_1+1}\cdots u^{i_j+1}\cdots u^{i_s+1}.
\end{equation}

Let $k=\sum_{j=1}^si_j.$ Using Lemma \ref{a_rLemma} (i) and the fact that  $ g^r x^s =  
\om^{-rs} x^s g^r $, (\ref{term}) becomes
$$\displaystyle a_{k}\omega^{-(\sum_{t< r}i_ri_t)}\Pi_{j=2}^s \om^{\sum_{l=1}^{j-1}i_j(i_l+1)}x^kg^{-(k+s)}\otimes u^{k+s}.$$

Simplifying, we have 
\begin{align}
\num
\omega^{-(\sum_{t<r}i_ri_t)}\Pi_{j=2}^s\omega^{\sum_{l=1}^{j-1}i_j(i_l +1)}&=&\omega^{-(\sum_{t<r}i_ri_t)}\omega^{\sum_{j=2}^s\sum_{l=1}^{j-1}(i_ji_l+i_j)} \ \ \ \ \ \ \ \ \ \ \ \ \ \ \\
														&=&\omega^{-(\sum_{t<r}i_ri_t)}\omega^{\sum_{j=2}^s(\sum_{l=1}^{j-1}i_ji_l)}\omega^{\sum_{j=2}^si_j(j-1)} \notag \\
														&=&\omega^{\sum_{j=2}^si_j(j-1)}, \ \ \ \ \ \ \ \ \ \ \ \ \ \ \ \ \ \ \ \ \ \ \ \ \ \ \ \ \ \ \ \ \  \notag											
\end{align}
since the first two powers of $\omega$ which appear have opposite exponents. 

Finally, since such a term arises whenever $i_1+\cdots +i_s=k,$ by ordering the terms according to powers of $x$ we have that

$$\rho(u)^s=\sum_{k=0}^{n-1} a_k \left(\sum_{\{0 \leq i_1, \ldots, i_s\leq k \ | \ \sum_{j=1}^si_j=k \}}\omega^{\sum_{j=2}^si_j(j-1)} \right) x^k g^{-(k+s)}\otimes u^{k+s}.$$
\end{proof}

Using Proposition \ref{power} with $s+1$ instead of $s$, we have all the components of our desired equation (\ref{identity}). Substituting them in (\ref{identity}), we may compare the coefficients on both sides: 
\begin{eqnarray*}
\sum_{s=0}^{n-1}\sum_{k=0}^{n-1}a_sa_k \left(\sum_{\{0 \leq i_1, \ldots, i_{s+1}\leq k \ | \ \sum_{j=1}^{s+1}i_j=k\} }\omega^{\sum_{j=2}^{s+1}i_j(j-1)} \right) x^sg^{-(s+1)}\otimes x^kg^{-(k+s+1)}\otimes u^{k+s+1}\\
= \sum_{m=0}^{n-1}\sum_{l=0}^ma_m\omega^{l(m-l)}{m \choose l}_\omega x^{m-l}g^{l-(m+1)}\otimes x^lg^{-(m+1)}\otimes u^{m+1} 
\end{eqnarray*}

By linear independence, the coefficients of each term on both sides should agree. Thus we have:

\begin{thm}\label{main}
Fix a primitive $n$th root of unity $\omega$ in $\Bk$, and let $A = A_n(\om)$ and $H = T_{n^2} (\om)$ be as above. 
Then $A$ is a right $H$-comodule algebra via the coaction $\rho$ in Equation (\ref{coaction}) $\iff$ for all pairs of natural numbers $0\leq k,s\leq n-1$, 
\begin{equation*}\sum_{\{0 \leq i_1, \ldots, i_{s+1}\leq k \ | \ \sum_{j=1}^{s+1}i_j=k \} }\ds{\omega^{\sum_{j=2}^{s+1}i_j(j-1)}}=
\left\{ 
\begin{array}{cl}
\ds{{k+s \choose k}_\omega} & \text{if $k+s<n$}\\
\ds{0} & \text{if $k+s\geq n$}.
\end{array}  \right. 
\end{equation*}
\end{thm}

%---------------------------------------------------4. A proof of the identities----------------------------------------------

\section{A proof of the identities}

In this section we give a direct combinatorial proof of the identities in Theorem \ref{main}. 
We thank Jason Fulman for pointing it out to us. 

We consider the expansion of $\ds \frac{1}{(1-z)(1-z\om)\cdots (1-z\om^s)}$ as a formal power series 
in the ring $\Bk[[z]]$. Write 
$$\ds \frac{1}{(1-z)(1-z\om)\cdots (1-z\om^s)}=\sum_{k\geq0}\beta_kz^k.$$ 

\begin{lem}
For each $k\geq 0,$ 
$$\beta_k=\sum_{\{0 \leq i_1, \ldots, i_{s+1}\leq k \ | \ \sum_{j=1}^{s+1}i_j=k \} }\ds{\omega^{\sum_{j=2}^{s+1}i_j(j-1)}}.$$ 
\end{lem}

\begin{proof}
We know that $$\sum_{k\geq 0}\beta_kz^k=\Pi_{l=1}^{s+1}(\frac{1}{1-z\om^{l-1}})=\Pi_{l=1}^{s+1}\left(\sum_{i_l\geq 0}(z\om^{l-1})^{i_l}\right).$$ 

Whenever $\ds \sum_{l=1}^{s+1}i_l=k,$ the last product gives a term $\ds z^k\om^{\sum_{l=2}^{s+1}i_l(l-1)},$ where the sum in the exponent starts at $l=2$ because $l-1=0$ for $l=1.$  Thus
$$\beta_k=\sum_{\{0 \leq i_1, \ldots, i_{s+1}\leq k \ | \ \sum_{l=1}^{s+1}i_l=k \} }\ds{\omega^{\sum_{l=2}^{s+1}i_l(l-1)}}$$
and thus the left hand side of the identity in Theorem  \ref{main} is the coefficient of $z^k$ in the power series. 
\end{proof}

\begin{thm} \label{comain} The identities in Theorem \ref{main} hold, for all $n > 1$, any given primitive $n^{th}$ root of unity $\omega$ in $\Bk$, and all pairs of natural numbers $0\leq k,s\leq n-1$.
\end{thm} 

\begin{proof}
We evaluate the coefficient $\beta_k$ in a different way, using Theorem 349 in \cite{HW} which states that given $\om\in \Bk,$ 
$$\ds \frac{1}{(1-z\om)(1-z\om^2)\cdots (1-z\om^j)}=1+z\om\frac{1-\om^j}{1-\om}+z^2\om^2\frac{(1-\om^j)(1-\om^{j+1})}{(1-\om)(1-\om^2)}+\cdots.$$
Replacing $z\om$ by $z$ we get
$$\ds \frac{1}{(1-z)(1-z\om)\cdots (1-z\om^{j-1})}=1+z\frac{1-\om^j}{1-\om}+z^2\frac{(1-\om^j)(1-\om^{j+1})}{(1-\om)(1-\om^2)}+\cdots$$
and if we choose $j=s+1$ then
$$\ds \frac{1}{(1-z)(1-z\om)\cdots (1-z\om^s)}=1+z\frac{1-\om^{s+1}}{1-\om}+z^2\frac{(1-\om^{s+1})(1-\om^{s+2})}{(1-\om)(1-\om^2)}+\cdots.$$
In particular, the coefficient $\beta_k$ of $z^k$ turns out to be
$$\frac{(1-\om^{s+1})\cdots (1-\om^{s+k})}{(1-\om)\cdots (1-\om^k)}={{k+s \choose k}_\om}.$$

Since $\beta_k$ is unique, both forms must agree and  
$$\sum_{\{0 \leq i_1, \ldots, i_{s+1}\leq k \ | \ \sum_{j=1}^{s+1}i_j=k \} }\ds{\omega^{\sum_{j=2}^{s+1}i_j(j-1)}}={{k+s \choose k}_\om}.$$

When $k+s\geq n$ with $0\leq k\leq n-1,$ one of the factors in the numerator \\
$(1-\om^{s+1})\cdots (1-\om^{s+k})$ is $1-\om^n=0$ while the denominator $(1-\om)\cdots (1-\om^k)\neq 0,$
making ${{k+s \choose k}_\om}=0$ as required.
\end{proof}

\begin{cor} \label{comod2} The algebra $A$ is an $H$-comodule algebra, via the coaction in Equation 
(\ref{coaction}).
\end{cor} 
%---------------------------------------------5. YD-module algebras------------------------------

\section{YD-module algebras and $H$-commutativity}

In this section we consider the (left, left) Yetter-Drinfel'd category $^H_H{\CY \CD}$. 
Recall that a module $M$ is in $^H_H{\CY \CD}$ if it is both a 
left $H$-module, a left $H$-comodule (via $\rho$), and 
\begin{equation} \label{YD} \num
h\cdot \rho(m) = \sum \rho(h_1 \cdot m) (h_2 \ot 1).
\end{equation}

 \cite[Prop 2.2(e)]{CFM} prove that our algebra $A = A_n$ is in ${^H_H{\CY \CD}}$ for 
 $H = T_{n^2} (\om)$, for all $n \leq 4$.  Here we show this for all $n$. We use a result from 
 \cite{CFM} which holds for any $H$ and any $A$:

\begin{lem} \cite[Lemma 2.10]{CFM} \label{genAH} Let $A$ be a left $H$-module and a left  
$H$-comodule.

(a) Let $M$ be an $H$-submodule of $A$. If the Yetter-Drinfel'd condition is satisfied for all 
$m \in M$ and all algebra generators of $H$ (from some chosen generating set), then it is 
satisfied for all $m \in M$ and all $H \in H$. 

(b) If $A$ is also an $H$-module algebra and an $H$-module coalgebra, and if the Yetter-Drinfeld 
condition holds for all $h \in H$ and all algebra generators of $A$  (from some generating set), 
then $A \in {^H_H{\CY \CD}}.$

\end{lem}

\begin{prop} \label{AYD} The algebra $A= A_n(\om)$ is in  ${^H_H{\CY \CD}}$ for the 
Taft algebra $H = T_{n^2}(\om)$, for all $n$.
\end{prop}

\begin{proof} By Corollary \ref{comod2}, $A$ is a left $H$-comodule, and so Lemma \ref{genAH} 
will apply. We use that $A$ is generated as an algebra by the $H$-submodule $M = k\{1,u\}$ and
 $H$ is generated as an algebra by the set $\{ g, x\}$. Thus $A$ will be in ${^H_H{\CY \CD}}$ 
 provided we show the Yetter-Drinfeld condition (\ref{YD}) when $a = u$ and either $h = g$ or $h = x$. 

First assume $h = g$. Then, using $\rho(u)$ as in (\ref{coaction}), 
\begin{eqnarray*}
g\cdot \rho(u)&=&\sum_{i=0}^{n-1}a_igx^ig^{-(i+1)}\otimes g\cdot u^{i+1}\\
		     &=&\sum_{i=0}^{n-1}a_i\om^{-i}x^ig^{-i}\ot w^{i+1}u^{i+1}\\
		     &=&\sum_{i=0}^{n-1}\om a_ix^ig^{-i}\ot u^{i+1}.
\end{eqnarray*}

On the other hand, since $g\cdot u=\om u$ 
\begin{eqnarray*}
\rho(g\cdot u)(g\otimes 1)&=&  \om \left( \sum_{i=0}^{n-1}a_ix^ig^{-(i+1)}\ot u^{i+1}\right)(g\ot 1) \\
				        &=&\sum_{i=0}^{n-1}\om a_ix^ig^{-i}\ot u^{i+1}.	                                         
\end{eqnarray*}
Thus the Yetter-Drinfel'd condition holds for $g$ and $u$.

Now assume that $h = x$. First, since $\Delta (x)=x\ot 1+g\ot x$ and $x\cdot u = 1$, it is easy to see 
that $ x\cdot u^{i+1}=(\sum_{j=0}^i\om^j)u^i$. Thus in (\ref{YD}), 

\begin{eqnarray*}
x\cdot \rho (u)&=&\sum_{i=0}^{n-1}a_ixx^ig^{-(i+1)}\ot u^{i+1} + \sum_{i=0}^{n-1}a_igx^ig^{-(i+1)}\ot x\cdot u^{i+1}\\
		      &=&\sum_{i=0}^{n-1}a_ix^{i+1}g^{-(i+1)}\ot u^{i+1}+ \sum_{i=0}^{n-1}\om ^{-i}a_ix^ig^{-i}\ot (\sum_{j=0}^i\om ^j)u^i\\
		      &=&1\ot 1+\sum_{i=1}^{n-1}\left( (\sum_{j=0}^i\om^{j-i}) a_i+a_{i-1}\right)x^ig^{-i}\ot u^i.
\end{eqnarray*}

On the other hand,
\begin{eqnarray*}
\sum \rho(x_1 \cdot m) (x_2 \ot 1)&=&\rho(x\cdot u) (1\ot 1)+\rho(g\cdot u)(x\ot 1)\\
&=& \rho(1)(1\ot 1)+\om\left( \sum_{i=0}^{n-1}a_ix^ig^{-(i+1)}\ot u^{i+1}\right)(x\ot 1)\\
								&=&1\ot 1 + \om\sum_{i=0}^{n-1}\om^{i+1}a_ix^{i+1}g^{-(i+1)}\ot u^{i+1}\\
								&=& 1\ot 1 + \sum_{i=1}^{n-1}\om^{i+1}a_{i-1}x^ig^{-i}\ot u^i,
\end{eqnarray*}
where in both cases the term corresponding to $i=n$ vanishes because $x^n=0.$ 
Thus for $A$ to be a Yetter-Drinfel'd module algebra, we need that 
$$(\sum_{j=0}^i\om^{j-i}) a_i+a_{i-1}=\om^{i+1}a_{i-1}, \text{ for all } 1\leq i\leq n-1.$$
However this holds by Lemma \ref{a_rLemma} (ii). 
%But \begin{eqnarray*}
%(\sum_{j=0}^i\om^{j-i}) a_i+a_{i-1}&=& \om^{-i}\frac{\om^{i+1}-1}{\om-1}a_i+a_{i-1}\\
						     %&=&\om^{-i}\frac{\om^{i+1}-1}{\om-1}(\om-1)^i\om^{\frac{i(i+1)}{2}}+(\om-1)^{i-1}\om^{\frac{(i-1)i}{2}}\\
						     %&=&(\om-1)^{i-1}\left( (\om^{i+1}-1)\om^{\frac{i(i+1)}{2}-i}+\om^{\frac{(i-1)i}{2}}\right)\\
						     %&=&(\om-1)^{i-1}\left( \om^{i+1}\om^{\frac{(i-1)i}{2}}\right)\\
						     %&=&\om^{i+1}a_{i-1}.
%\end{eqnarray*}
\end{proof}

\cite{CFM} also study when $A$ is commutative as an algebra in the category $^H_H{\CY \CD}$. 

Recall that for any braided monoidal category $\CC$, with braiding $\tau: V \ot W \to W \ot V$ for $V, W \in \CC$, an algebra $A$ in $\CC$ is {\it commutative in $\CC$} if for all $a, b \in A$, 
\begin{equation}\num \label{comm}m_A (a \ot b) = m_A \circ \tau (a \ot b).
\end{equation}

Several authors have considered this generalized commutativity. In particular Cohen and Westreich considered the case when $\CC$ is the module category of a quasi-triangular Hopf algebra  in \cite{CW}. 

In our situation $^H_H{\CY \CD}$ has the structure of a braided monoidal category, as follows: for two modules 
 $M, N \in ^H_H{\CY \CD}$, the braiding is given as follows \cite{Y}: 
$$\tau: M \ot N \to N \ot M \  \  \ \text{via} \ \ \  m \ot n \mapsto \rho(m) ( n \ot 1) = \sum (m_{-1} \cdot n) \ot m_0.$$ 

Thus an algebra $A$ in $^H_H{\CY \CD}$ is commutative in $^H_H{\CY \CD}$ if 
\begin{equation}\num 
ab = \sum (a_{-1} \cdot b) a_0.
\end{equation} 

\begin{cor} For the given algebra $A_n = k[u]$ and $H = T_{n^2}(\om)$, $A$ is 
commutative in $^H_H{\CY \CD}$, for any $n$.
\end{cor}

\begin{proof} It is shown in \cite[Prop 2.2(e)]{CFM} that if $A_n$ is in $^H_H{\CY \CD}$, then it is commutative in $^H_H{\CY \CD}$. In fact their argument uses only that $A_n$ is an $H$-module 
$H$-comodule algebra; again it suffices to check on generators of $A$ and of $H$.
\end{proof}

%---------------------------------------------------6. Dualization----------------------------------------------

\section{The dual action}

In this section, for the sake of completeness, we sketch the approach of \cite{MS} for the action of 
$H^*$ on $A$. As noted in the introduction, it is shown there that $A$ is a $D(H)$-module algebra 
(and thus a Yetter-Drinfeld module algebra). 

The Taft Hopf algebras $H =  T_{n^2}(\om)$ are known to be self-dual; thus we may write 
\begin{equation}\num\label{taft*}
H^* = \Bk\langle G,X \vert\; G^n = \eps, X^n = 0, XG=\omega GX\rangle,
\end{equation}
where $\Delta(G)=G\otimes G$, $\Delta(X)= X\otimes \eps + G\otimes X$, 
$\lan G,1\ran =1$, and $\lan X,1\ran=\eps_{H^*}(X)=0$. 

The dual pairing between $H$ and $H^*$ is determined by 
\begin{equation}\num\label{pairing}
\lan G,g\ran= \omega^{-1},\;\lan G,x\ran =0,\; \lan X,g\ran= 0,\;{\rm and} 
\;\lan X,x\ran=1.
\end{equation}

\begin{lem} \label{gen}As an algebra, $D(H)$ is generated by $\{ x, g, X, G \}$. The relations among 
these generators, in addition to the relations in $H$ and $H^*$, are as follows:
$$ gG=Gg, \  \ xG = \om^{-1}Gx, \  \ Xg = \om^{-1} gX, \  \text{ and } \  xX-Xx = G-g.$$
\end{lem}

One may check that  $(H^*)^{cop} = \Bk\langle  G^{-1}, XG^{-1} \rangle \subset D(H),$
and that these generators give the usual relations in $D(H)$. The generators given in Lemma \ref{gen} are used since 
$X$ and $x$ behave similarly when acting as skew derivations. \cite{MS} then use properties of higher skew derivations and the relations in Lemma \ref{gen} to prove:

\begin{thm} \cite[Theorem 4.5]{MS} \label{actions} Let $H = T_{n^2}(\om)$ be the Taft Hopf algebra and $H^*$ its dual as above. 
Then $A = A_n$ becomes a $D(H)$-module algebra via the following: 

(a) $g\cdot u = \om u$ and $G \cdot u = \om^{-1}u$, and 

(b) $x\cdot u = 1$ and $X \cdot u = ( \om^{-1} -1) u^2.$ 
\end{thm} 

To see that $A $ is an algebra in the category 
$^H_H{\CY \CD}$ of left, left Yetter-Drinfeld modules, one may use a theorem of Majid \cite{Mj} that $D(H)$-modules maybe be identified with $^H_H{\CY \CD}$-modules. The only difficulty remains in showing that dualizing the left $(H^*)^{cop}$-action in Theorem \ref{actions} to a left $H$-comodule action gives the desired coaction. 

\begin{thm} \cite[Theorem 5.7]{MS} Let $H = T_{n^2}(\om)$, $A = A_n$, and the $H$-action on $A$ 
be as described in Theorem \ref{actions}.  Then there is a unique left $H$-comodule algebra structure 
$\rho$ on $A$ such that $A$ is in $^H_H{\CY \CD}$, given by
$$
\rho(u)= \sum_{m=0}^{n-1}a_m x^m g^{-(m+1)}\otimes u^{m+1},
$$
where the coefficient $a_m$ is given by
$$
a_m= ((1-\omega^{-1})\omega)^m \omega^{\frac{m(m+1)}{2}}=(\om -1)^m \omega^{\frac{m(m+1)}{2}}.
$$
\end{thm} 

This coefficient $a_m$ is exactly our coefficient in Definition (\ref{a_i}), and so the coaction in (6.5) 
is exactly our coaction in Equation (\ref{coaction}). Thus \cite[Theorem 5.7]{MS} gives an alternate 
proof of Corollary \ref{comod2} and Proposition \ref{AYD}.

\vs

\begin{center}
ACKNOWLEDGMENT
\end{center}
The authors wish to thank Jason Fulman for suggesting the proof of 
Theorem \ref{comain}.

%---------------------------------References--------------------------------------------------

\end{document}